\documentclass{article}
\usepackage[utf8]{inputenc}
\usepackage{geometry}
\usepackage{amsfonts,amsmath,amsthm,amssymb}
\usepackage[dvipsnames]{xcolor}
\usepackage{hyperref}
\usepackage{array,float}
\usepackage{graphicx}
\usepackage{caption}
\usepackage{enumitem}
\usepackage{mathtools}
\usepackage{tikz-cd}
\usepackage[noadjust]{cite}
\usepackage{graphicx}
\usepackage[normalem]{ulem}
\usepackage{cleveref}
\usepackage{comment,verbatim}

\theoremstyle{plain}
\newtheorem{theorem}{Theorem}[section]
\newtheorem{proposition}[theorem]{Proposition}
\newtheorem{lemma}[theorem]{Lemma}
\newtheorem{conjecture}[theorem]{Conjecture}
\newtheorem{remark}[theorem]{Remark}

\newtheorem{corollary}[theorem]{Corollary}
\newtheorem{example}[theorem]{Example}

\newtheorem{introtheorem}{Theorem}
\newtheorem{introcorollary}[introtheorem]{Corollary}

\theoremstyle{definition}
\newtheorem{definition}[theorem]{Definition}

\newcommand{\R}{\mathbb{R}}
\newcommand{\C}{\mathbb{C}}
\newcommand{\Z}{\mathbb{Z}}
\newcommand{\Q}{\mathbb{Q}}
\newcommand{\E}{\mathcal{O}}

\DeclareMathOperator{\vol}{vol}
\DeclareMathOperator{\Gal}{Gal}

\DeclareMathOperator{\reg}{reg}
\DeclareMathOperator{\Log}{Log}

\title{Classification of totally real number fields via their zeta function, regulator, and log-unit lattice}
\author{Jos\'e Cruz }
\date{\today}

\begin{document}

\maketitle
\abstract{In this paper, assuming the weak Schanuel Conjecture (WSC), we prove that for any collection of pairwise non-arithmetically equivalent totally real number fields, the residues at $s=1$ of their Dedekind zeta functions form a linearly independent set over the field of algebraic numbers. As a corollary, we obtain that, under WSC, two totally real number fields have the same regulator if and only if they have the same class number and Dedekind zeta function. We also prove that, under WSC, the isometry and similarity classes of the log unit lattice of a real Galois number field of degree $[K:\Q]\geq 4$, characterize the isomorphism class of said field. All of our results follow from establishing that, under WSC, any Gram matrix of the log unit lattice of a real Galois number field yields a generic point of certain closed irreducible $\Q$-subvariety of the space of symmetric matrices of appropriate size.}
\tableofcontents

\section{Introduction}
Two number fields are said to be arithmetically equivalent if their Dedekind zeta functions are identical. As a consequence of the class number formula, arithmetically equivalent number fields satisfy that the product $h_K\,\reg(K)$ of their class number and regulator is the same. However, De Smit and Perlis \cite{deSmitPerlis1994} found examples of arithmetically equivalent number fields with distinct class numbers, and therefore distinct regulators. This suggests that the regulator contains information about its number field that is not detected by the zeta function. Motivated by De Smit and Perlis' findings, we decided to search for totally real number fields with the same regulator. All the non isomorphic examples we found turned out to be arithmetically equivalent.


 Now, since the regulator is computed by the covolume of the log unit lattice and the signature of its number
field (see \Cref{definition: regulator}), it is natural to wonder if having the same regulator implies having isometric log unit lattices. It turns out that for all of the examples mentioned above, their log unit lattices are not isometric, furthermore, their log unit lattices are not even similar. 

In this paper, we address the following questions which are inspired by the above phenomena. All of our answers are conditional on the Weak Schanuel Conjecture. See \Cref{conjecture: independence of logs}.
   
\begin{enumerate}
    \item\label{question: same regulator}  Are two totally real number fields with the same regulator arithmetically equivalent? 
    \item\label{question: same shape}  Are two totally real number fields with isometric (or similar) log unit lattices isomorphic? 
\end{enumerate} 

We will see in the sequel that the above questions are deeply connected to the following problem:
\begin{enumerate}
    \item[3.]\label{question: linear independence} Let $K_1,\dots, K_n$ be totally real number fields that are pairwise not arithmetically equivalent, and let $\zeta_{K_i}(s)$, for $i=1,\dots,n$, be the Dedekind zeta function of $K_i$. Is the set\, $\mathfrak{res}:=\{\mathrm{res}_{s=1}\zeta_{K_i}(s)\}_{i=1}^n$\, linearly independent over the field of all algebraic numbers? 
\end{enumerate}
Let us mention that Question \ref{question: same shape} was already asked for the similarity class of totally imaginary $D_6$-sextic number fields by \cite{BIRS-25frg504-2025}, and their findings depend on the signature of the cubic subfield of the sextic. We will give more details on these results in \Cref{subsubsection: log unit lattice}.

\subsection{Main results}
We now state our answers to the above questions.
\subsubsection{On number fields with the same regulator and the linear independence of the residues}
Let $\overline{\Q}\subset \C$ be the field of all algebraic numbers. If we think of $\C$ as a $\overline{\Q}$-vector space, the class number formula yields the following identity in the projective space $\mathbb{P}_{\overline{\Q}}(\C)$:
\[
   [\mathrm{res}_{s=1}\zeta_K(s)]=[\reg(K)],\quad\text{for any totally real number field $K/\Q$}.
\]
 Our first result states that for totally real number fields, the zeta function is completely determined by the projective point $[\mathrm{res}_{s=1}\zeta_K(s)]\in\mathbb{P}_{\overline \Q}(\C)$, or equivalently, by the projective point $[\reg(K)]$.
\begin{introtheorem}[\Cref{cor: Number fields with the same regulator}]\label{introtheorem: number field with the same regulator}
    Assume the weak Schanuel conjecture (\Cref{conjecture: independence of logs}). Let $K_1$ and $K_2$ be two totally real number fields. If $\mathrm{res}_{s=1}\zeta_{K_1}(s)$ and $\mathrm{res}_{s=1} \zeta_{K_2}(s)$ are linearly dependent over the field of algebraic numbers $\overline{\Q}$, then $\zeta_{K_1}=\zeta_{K_2}$. Furthermore, $\reg(K_1)=\reg(K_2)$ if and only if $\zeta_{K_1}=\zeta_{K_2}$ and  $h_{K_1}=h_{K_2}$, where $h_{K_1}$ and $h_{K_2}$ are the class numbers of $K_1$ and $K_2$, respectively.
\end{introtheorem}
This implies that for totally real number fields, the regulator characterizes both the Dedekind zeta function and the class number of its number field. Note, however, that there are examples of arithmetically equivalent number fields with same regulator, and therefore same class number, but with non-isomorphic ideal class groups. Hence, the regulator cannot capture the isomorphism class of the ideal class group.

\Cref{introtheorem: number field with the same regulator} is a corollary of the main conditional theorem of this paper, which establishes both the $\overline{\Q}$-linear independence at $s=1$ of distinct Dedekind zeta functions of totally real number fields and the $\overline{\Q}$-linear independence of the corresponding regulators. 
\begin{introtheorem}[\Cref{Theorem: RegulatorArithmeticInvariant}]\label{introtheorem: linear independence}
     Assume \Cref{conjecture: independence of logs} and let $K_1,\, K_2,\,\dots,\,K_n$ be totally real number fields all distinct and different from $\Q$ such that $\zeta_{K_i}\neq\zeta_{K_j}$ for every $i\neq j$. Then, the sets  $\mathfrak{reg}:=\{\reg(K_i)\}_{i=1}^n$ \text{and} $\mathfrak{res}:=\{\mathrm{res}_{s=1}\zeta_{K_i}(s)\}_{i=1}^n$ are linearly independent subsets of $\C$ over $\overline{\Q}$.  
\end{introtheorem}

\subsubsection{On the log unit lattice}

Let us now state our results concerning the log unit lattice. Since the covolume of the log unit lattice is the regulator multiplied by an algebraic number (see \Cref{definition: regulator}), it follows that the isometry class of the log unit lattice determines the projective point $[\reg(K)]\in \mathbb P_{\overline{\Q}}(\C)$. As a direct consequence of \Cref{introtheorem: number field with the same regulator}, we find that, for totally real number fields, the isometry class conveys even more information: 
\begin{introcorollary}[\Cref{cor: isometry}]\label{introcor: isometry}
    Assume \Cref{conjecture: independence of logs}. Let $K_1$ and $K_2$ be two totally real number fields and let $\Lambda_{K_1}$ and $\Lambda_{K_2}$ be the log unit lattices of $K_1$ and $K_2$, respectively. If $\Lambda_{K_1}$ and $\Lambda_{K_2}$ are isometric, then $\reg(K_1)=\reg(K_2)$. In particular, $\zeta_{K_1}=\zeta_{K_2}$ and $h_{K_1}=h_{K_2}$, where $h_{K_1}$ and $h_{K_2}$ are the class numbers of the corresponding number fields.
\end{introcorollary}
We conjecture that two totally real number fields with isometric log unit lattices are isomorphic, but the proof of this seems untractable at the moment. Nevertheless, we show in \Cref{example: nonisometric} a pair of totally real number fields with the same regulator but non-isometric and even non-similar log unit lattices. Hence, the isometry class of the log unit lattice conveys more information than the regulator about the isomorphism class of its field. This example can be found in \cite{mantilla2015weak} and was suggested to the author of this work by Mantilla-Soler. 

As for the similarity class, we obtained information only in the real Galois case. In this case, when one avoids degenerate cases, equality of the similarity classes implies isomorphism of the number fields.
\begin{introtheorem}\label{introtheorem: sameshape}
    Assume the weak Schanuel conjecture \Cref{conjecture: independence of logs}.
    Let $K_1/\Q$ and $K_2/\Q$ be two real Galois number fields. For $i=1,2$, let $n_i:=[K_i:\Q]$ and $G_i:=\Gal(K_i/\Q)$. If $n_1,n_2>3$, then $[\Lambda_{K_1}]=[\Lambda_{K_2}]$\, if and only if $K_1\cong K_2$.
\end{introtheorem}
Note that arithmetically equivalent Galois number fields must be isomorphic, thus, if we restrict ourselves to real Galois number fields, \Cref{introtheorem: number field with the same regulator} and \Cref{introtheorem: sameshape}  already imply that the regulator and the isometry class of the log unit lattice completely determine the isomorphism class of their number field. 

 \subsubsection{On the genericity of the log unit lattice} We now present The main technical tool of the paper: \Cref{theorem: log-lats are generic}. Let us motivate the statement of the result. First, fix a group $G$ and consider the family of number fields $\mathcal{F}(G):=\{K:\ K/\Q \text{ is Galois, $\Gal(K/\Q)\cong G$ and $K\subset\R$}\}$. \Cref{theorem: log-lats are generic} is the result of the following informal procedure for studying the log unit lattices of number fields in the family $\mathcal{F}(G)$. 
\begin{enumerate}
    \item Find a suitable closed subvariety $\mathcal V_G$ of the space of complex symmetric matrices $\mathrm{Sym}^{n-1}(\C)$ that contains a representative of the Gram matrix, modulo rational changes of basis, of the log unit lattice of every member of the family $\mathcal{F}(G)$. This is inspired by the results of \cite{HHV}. 
    \item Use the Weak Schanuel Conjecture to conclude that no Zariski closed subset defined over $\Q$ of $\mathrm{Sym}^{n-1}(\C)$ contains a Gram matrix  of the log unit lattice of a member of $\mathcal{F}(G)$. This is inspired on \cite{RNT}.
\end{enumerate}
Now consider the quotient ring $R:=\Z[G]/(N_G)$ and let $R_\C:=\C[G]/(N_G)$. Then for any $K\in \mathcal{F}(G)$, the action of the Galois group $G$ endows the log unit lattice $\Lambda_K$ with the structure of a left $R$-module. It turns out that the variety $\mathcal{V}_G$ we are looking for is given by the space $\mathrm{Sym}^G(R_\C)$ of symmetric $G$-invariant complex bilinear forms on the $\C$-vector space $R_\C$. The following theorem is a simplified version of \Cref{theorem: log-lats are generic}.

\begin{introtheorem}\label{introtheorem: genericity}
    Assume Conjecture \ref{conjecture: independence of logs}. Let $K/\Q$ be a real Galois number field with Galois group $G$ and let $\Lambda_K$ be its log unit lattice. 
    let $\mathcal{V}_G$ be as above. Suppose $u\in \E_K^\times$ is a weak Minkowski unit\footnote{A unit $u\in \E_{K}^\times$ is said to be a weak Minkowski unit, if the unit and its conjugates generate a finite index subgroup of $\E_K^\times$.} (see \Cref{Theorem: existence of weak Minkowski unit}). If $v=\Log(u)$, then, the bilinear form $\mathfrak{Gr}_v:R_\C\times R_\C\rightarrow \C,\,\mathfrak{Gr}_v(\alpha,\beta)=\langle \alpha v,\beta v\rangle$,
    is not contained in any Zariski-closed proper subset defined over $\overline{\Q}$ of $\mathcal{V}_G$.
\end{introtheorem}
Let us give the intuition on how we use \Cref{introtheorem: genericity} for the proofs of \Cref{introtheorem: linear independence} and \Cref{introtheorem: sameshape}. The idea is to convert a linear relation between the residues of distinct zeta functions, and the assumption of similarity of two log unit lattices coming from non-isomorphic number fields, into algebraic conditions satisfied by some Gram matrix of the log unit lattice of the Galois closure of the number fields involved. One then obtains that these algebraic conditions are satisfied by a generic point in the corresponding irreducible algebraic variety $\mathcal{V}_G$ and hence satisfied by any point in the variety, which is extremely useful to reach the desired conclusions.

 \subsection{Previous work}
 We now provide an overview of the existing literature related to the questions posed in the introduction.
\subsubsection{On number fields with the same regulator and arithmetic equivalence:}
 A natural source of number fields with the same regulator is the family of CM-fields. See \cite{iizuka2025infinite} for the construction of infinitely many pairs of CM quartic fields with the same discriminant, regulator, and class number. In fact, Greither and Zaimi proved in \cite{greither2017cm} that for any totally real number field $K$, there are only finitely many CM-extensions $L/K$ of the form $L=K(\epsilon)$ for $\epsilon$ a unit of $K$. From their result, it is not hard to deduce that the infinite family $\{L\ :\ L/K\text{ is a CM-extension} \}$ only has finitely many number fields with different regulators. 

 However, if one drops the CM condition, examples of number fields with the same regulator are now scarce. It seems to be the case that non-CM number fields with the same regulator have to be arithmetically equivalent; but, to the author's knowledge, there is no treatment of this in the literature. 
 
 On the other hand, there are two refinements of the condition of arithmetic equivalence that capture both the class number and the regulator. The first refinement, due to Prasad \cite{prasad2017refined}, is called \emph{integral equivalence}. It determines the Dedekind zeta function, the idéle class group, and, therefore, the regulator. The second refinement was introduced by Sutherland in \cite{Sutherland2021Stronger}; it is called \emph{locally integral equivalence}. Sutherland proved that two locally integrally equivalent number fields must have the same Dedekind zeta function, class number, and therefore regulator. He also proved that integral equivalence implies locally integral equivalence. It would be interesting for a future project to investigate if these refinements can determine the isometry/similarity class of the log unit lattice of a totally real number field.

\subsubsection{On the log unit lattice:}\label{subsubsection: log unit lattice}
There has been a recent surge of literature concerning the log-unit lattice (\cite{cusick1990units},\cite{kessler1991minimum},  \cite{david2018dirichlet}, \cite{TPS2021},\cite{cruz2020well}, \cite{HHV}, \cite{corso2025unboundednessshapesunitlattices} \cite{dang2025densityshapesperiodictori}, \cite{RNT}, \cite{BIRS-25frg504-2025}). Currently, the main problems that govern the subject are the following:
\begin{enumerate}
    \item Find log unit lattices with prescribed geometric properties, such as orthogonality and well-roundedness. 
    \item Understand how log unit lattices of a prescribed family of number fields distribute in the space of similarity classes of lattices of the appropriate rank.
\end{enumerate}
All the aforementioned references give results that elucidate parts of the above problems for certain families of number fields. We now highlight the results which inspired this work.
\begin{itemize}
    \item In \cite{HHV}, the authors show that, for a fixed prime number $p$, the similarity class of the log unit of a dihedral $D_p$-number field of signature $[1,(p-1)/2]$ is contained in a preperiodic orbit of a torus in the space of similarity classes of lattices of rank $(p-1)/2$. 
    \item In \cite{RNT} the authors prove that the similarity class of the log unit lattice of any $D_4$-quartic field of signature $[2,1]$ lies in the border of the moduli space $\mathcal{S}_2\subset \C$ that parametrizes the similarity classes of rank 2 lattices. They moreover prove that for every single one of these number fields, the corresponding complex number describing the similarity class of the log unit lattice is transcendental. 
    \item In the work in progress \cite{BIRS-25frg504-2025} the authors show that for totally imaginary $D_6$-sextic fields there is a dichotomy determined by the signature of the cubic subfield of the sextic. On one hand, if the cubic subfield is complex, then the similarity class of the log unit lattice of the sextic is represented by a trascendental number in the border of $\mathcal{S}_2$. On the other hand, if the cubic subfield is real, then the corresponding point in $\mathcal{S}_2$ is also trascendental but lies in the interior of $\mathcal{S}_2$. 
    
    The authors show that the dichotomy goes even further; it dictates when is the similarity class a complete invariant of the sextic. They prove there are infinitely many non-isomorphic totally imaginary $D_6$-sextic number fields with real cubic subfield with similar log unit lattice, while, on the other hand, if we restrict to the family of $D_6$-sextic number fields with signature $[0,3]$ admitting a complex cubic subfield, then the similarity class of the log unit lattice is a complete invariant: two number fields in this family are isomorphic if and only their log unit lattice is similar.
\end{itemize}
All of the above results start by describing where the similarity class of a log unit lattice of a member of the prescribed family of number fields lives. Then \cite{RNT} and \cite{BIRS-25frg504-2025} prove that the corresponding similarity class yields a transcendental point in the space. The methodology of \Cref{introtheorem: genericity} transports the same ideas to the case of real Galois number fields, and replaces the word transcendental with generic. 

\subsubsection{On the transcendental nature of special values of L-functions:}

The study of the transcendence of values of $L$-functions has a rich history that goes back to Euler. See \cite{baker1973problem}, \cite{gun2011transcendental}, \cite{gun2012linear} and \cite{LUTES20131000}. Perhaps the most reminiscent result to our conditional \Cref{introtheorem: linear independence}, is the following one by Baker \cite{baker1973problem}: For any set of even Dirichlet characters $\{\chi_i\}_{i=1}^n$, the set containing the corresponding $L$-functions evaluated at $s=1$:\, $\{L(1,\chi_i)\}_{i=1}^n$, is a $\overline{\Q}$-linearly independent set.

It is natural to ask if the same result holds for general Artin L-functions. In fact, In \cite[Theorem 5.3]{gun2011transcendental} Gun et al prove that for any Galois representation $\pi:\Gal(\overline{\Q}/\Q)\rightarrow \mathrm{GL}_n(\C)$ such that $\pi$ has no nontrivial irreducible constituents, the weak Schanuel conjecture and the Stark conjecture imply the transcendence of the number $L(1,\pi)$. Given that the Dedekind zeta function is the L-function of certain Galois representation $\rho_K:\Gal(\overline{\Q}/\Q)\rightarrow\mathrm{GL}_{[K:\Q]}(\C)$ (see \cite{MantillaSoler2019} for more details), \Cref{introtheorem: linear independence} gives more instances of L-functions for which an analogue of Baker's result holds.



\subsection{Outline of the paper:} We now give the structure of the paper. In \Cref{section: preliminaries} we give all the necessary definitions and auxiliary results needed for the proofs of our main results. We then proceed in \Cref{section: genericity} with the proof of the main technical tool of the paper: \Cref{introtheorem: genericity}. Afterwards, in \Cref{section: linear independence} we introduce some auxiliary lemmas from algebraic geometry and arithmetic equivalence, and then we prove \Cref{introtheorem: linear independence} and \Cref{introtheorem: number field with the same regulator}. Finally, in \Cref{section: isometry and similarity}, we prove \Cref{introcor: isometry} and prove \Cref{introtheorem: sameshape}.
\section{Acknowledgments}
The author of this work is grateful to David Karpuk for suggesting the idea of studying log unit lattices of number fields, to Fatemeh Jalalvand for computing the minimum of the lattices in \Cref{example: nonisometric} and to Guillermo Mantilla-Soler for suggesting \Cref{example: nonisometric}. The author is also grateful to Kristaps Balodis, Erik Holmes, Fatemeh Jalalvand, Enrique Nunez Lon-Wo, Renate Scheidler and Ha Tran for fruitful discussions related to this paper.

\section{Preliminaries}\label{section: preliminaries}
\subsection{Lattices}\label{subsection: lattices}
We now introduce some basic definitions that are central to lattice theory and are easily found in any introductory text on the subject, for example \cite{martinet2013perfect}.  We recall them here for the sake of completeness and presentation. From now on, $(V,\langle\cdot,\cdot\rangle)$ denotes a finite-dimensional real inner product space.

\begin{definition}
A subgroup $\Lambda\subset V$ is said to be a \emph{lattice} if it is the $\Z$-span of some linearly independent subset $B\subset V$. The rank of $\Lambda$ is its rank as an abelian group $rk_\Z(\Lambda)$; it coincides with the cardinality of any linearly independent set spanning $\Lambda$. We will say that $\Lambda\subset V$ is of full-rank, if $rk_\Z(\Lambda)=\dim(V)$, or equivalently if any basis of $\Lambda$ is an $\R$-basis for $V$.
\end{definition}

\begin{definition}
    Let $\Lambda$ be a lattice and let $B=\{v_i\}_{i}$ be a basis of $\Lambda$. The Gram matrix of $\Lambda$ associated to the basis $B$ is the matrix:
\[
\mathrm{Gr}_B:=(\langle v_i,v_j\rangle)_{i,j}.
\]
The determinant $\det(\Lambda)$ is defined as the determinant of any Gram matrix of $\Lambda$ and the covolume is defined as $\sqrt{\vol(\Lambda)}$.
\end{definition}
We now define the notions of isometry and similarity, which will be fundamental for what follows.
\begin{definition}
    Given two lattices $\Lambda_1$ and $\Lambda_2$, we will say that $\Lambda_1$ is \emph{isometric} to $\Lambda_2$ if there exist $B_1\subset \Lambda_1$ and $B_2\subset\Lambda_2$ bases of $\Lambda_1$ and $\Lambda_2$, respectively, such that
    \[
    \mathrm{Gr}_{B_1}=\mathrm{Gr}_{B_2}.
    \]
    We will say that $\Lambda_1$ is similar to $\Lambda_2$ if there exist $B_1\subset \Lambda_1$ and $B_2\subset\Lambda_2$ bases of $\Lambda_1$ and $\Lambda_2$, respectively, such that  
    \[
    \mathrm{Gr}_{B_1}=\lambda\mathrm{Gr}_{B_2},\quad\text{for some $\lambda\in\R$}.
    \]
    Similarity and isometry define equivalence relations. We will denote the similarity class of a lattice $\Lambda$ by $[\Lambda]$, and therefore $\Lambda_1$ and $\Lambda_2$ are similar if and only if $[\Lambda_1]=[\Lambda_2]$. 
\end{definition}

\subsection{Log-unit Lattices}\label{subsection: loglats}
Let $K$ be a number field of degree $[K:\Q]=r+2s$, where  $r$ and $s$ denote the number of real and conjugate pairs of complex embeddings of $K$, respectively. We denote the ring of integers of $K$ by $\E_K$ and its corresponding unit group by $\E_K^\times$. 

\begin{definition}
    Let $\mathcal{H}_\R^{r+s-1}\subset \R^{r+s}$ denote the hyperplane defined by the equation $\sum_{i=1}^{r+s}x_i=0$, where the $x_i$ denote the coordinates of $\R^{r+s}$. The log embedding of $K$ is the map $\Log: \E_K^\times\rightarrow \mathcal{H}_\R^{r+s-1}$ given by:
     \[
    \Log(u)=(
         \log|\sigma_1(u)|,\dots,\log|\sigma_r(u)|,2\log|\tau_1(u)|,\dots,2\log|\tau_s(u)|).
    ).
    \]
    The \emph{log-unit lattice} $\Lambda_K$ of $K$ is the image of the map $\Log$.
\end{definition} 
The classical proof by Minkowski of Dirichlet's unit Theorem, shows that $\Lambda_K$ is a full-rank lattice in $\mathcal{H}_\R^{r+s-1}$. Note that $\Lambda_K$ depends on the choice of ordering of the embeddings of $K$, but a different choice of ordering produces an isometric lattice, so we will ignore this dependence from now on.
\begin{definition}\label{definition: regulator}
    The regulator $\reg(K)$ of $K$ is defined to be the following quantity:
    \[
    \reg(K)=\sqrt{\frac{\det(\Lambda_K)}{r+s}}.
    \]
\end{definition}
We now finish this subsection with the following lemma which shows that, for an extension $L/K$ of totally real number number fields, the usual inclusion of $K\hookrightarrow L$ yields a sublattice of $\Lambda_L$ that is similar to $\Lambda_K$. 

\begin{lemma}\label{lemma: including a log unit lattice}
	Let $N/K$ be an extension of totally real number fields. Let $\iota$ be the unique map that makes the following diagram commute:
    \[\begin{tikzcd}
	{O_K^\times} & {O_N^\times} \\
	{\Lambda_K} & {\Lambda_N}
	\arrow[hook, from=1-1, to=1-2]
	\arrow["{\Log_K}"', from=1-1, to=2-1]
	\arrow["{\Log_N}", from=1-2, to=2-2]
	\arrow["\iota"', dashed, from=2-1, to=2-2]
\end{tikzcd}\]
Then $[\Lambda_K]=[\iota(\Lambda_K)]$. Moreover, 
\[
\frac{\det(\Lambda_K)}{\det(\iota(\Lambda_K))}\in\Q.
\]
\end{lemma}
\begin{proof}
    Note that for any $v,w\in\Lambda_K$:
    \[
    \langle \iota(v),\iota(w)\rangle=\sqrt{[N:K]}\langle v,w\rangle.
    \]
    It follows that the lattices $\sqrt{[N:K]}\Lambda_K$ and $\iota(\Lambda_K)$ are isometric. The result easily follows from this observation.
\end{proof}
\subsection{The ring $R[G]$}\label{subsection: the ring R}
In this subsection $G$ denotes a finite group and $F$ denotes a field of characteristic 0. The integral group ring of $G$ is denoted by $\Z[G]$.
\begin{definition}
    Let $N_G=\sum_{g\in G}g\in \Z[G]$. We define the ring $R[G]$ to be the quotient ring $R[G]:=\Z[G]/(N_G)$, where $(N_G)$ is the two-sided ideal generated by $N_G$. If the group $G$ is clear from the context, we will omit $G$ and write $R$ instead of $R[G]$. We write $R_F[G]$ for the extension of scalars of $R$ from $\Z$ to $F$, so that $R_F[G]\cong F[G]/(N_G)$. Just as before, if $G$ is clear from the context, we will omit $G$ and write $R_F$ instead of $R_F[G]$.
\end{definition}
Note that for a Galois number field $L$ with Galois group $G=\Gal(L/\Q)$, the group $G$ acts orthogonally on the log unit lattice $\Lambda_L$ of $L$, since it acts by permuting the coordinates of $\Lambda_L$ inside $\mathcal{H}_\R^{r+s-1}$.
Furthermore,  $\Lambda_L$ has the structure of an $R[G]$-module and we have the following theorem.
\begin{theorem}[See Theorem 3.26 of \cite{ElementaryandAnalyticofANT}]\label{Theorem: existence of weak Minkowski unit}
    Let $L/\Q$ be a Galois number field of degree $[L:\Q]=r+2s$, where $r$ and $s$ are the number of real and complex embeddings of $L$, respectively, and let $G:=\Gal(L/\Q)$ be its Galois group. There exists a weak Minkowski unit $u\in \E_L^\times$, that is, a unit such that the set $\{u^\tau\}_{\tau\in G}$ contains $r+s-1$ independent units. 
\end{theorem}
From the above theorem, we get that  $\Q\otimes_\Z\Lambda_L$ is a left cyclic $R_\Q$-module, and when $L/\Q$ is real we get the following result:

\begin{corollary}\label{cor: module structure of log unit lattice}
    Let $R_\Q=\Q[G]/(N_G)$ and let $L/\Q$ be a real Galois number field with Galois group $G:=\Gal(L/\Q)$ and let $\Lambda_L$ be its log-unit lattice. Then
    \begin{enumerate}
        \item[(i)] $\Lambda_L$ is a $G$-module, and $G$ acts on $\Lambda_L$ by automorphisms of $\Lambda_L$
        \item[(ii)] If $u\in \E_L^\times$ is a weak Minkowski unit and $v=\Log_L(u)$, then the $R_{\Q}$-linear map $\phi_v:R_\Q\rightarrow \Q\otimes_\Z\Lambda_L$ determined by $\phi(1)=v$ is an isomorphism of left $R_\Q$-modules.
    \end{enumerate}
\end{corollary}
\begin{proof}
    By \Cref{Theorem: existence of weak Minkowski unit}, the map $\phi_v$ is a surjective map of vector spaces of the same dimension. The result follows.
\end{proof}
Let $L/K$ be an extension of totally real number fields. From now on we will view $\Lambda_K$ as a sublattice of $\Lambda_L$ via \Cref{lemma: including a log unit lattice}. With this in mind, we can state the following.
\begin{proposition}\label{prop: norm of weak Mink unit}
Let $L/\Q$ be a Galois number field with Galois group $G:=\Gal(L/\Q)$ and log unit lattice $\Lambda_L$. Fix $v=\Log(u)$ for some weak Minkowski unit $u\in \E_L^\times$. Let $K\subset L$ be a subfield, let $\Lambda_K$ be the log unit lattice of $K$ and let $H:=\Gal(L/K)$. Then
\begin{enumerate}
    \item[(i)] if $\phi_v:R_\Q[G]\rightarrow \Q\otimes_\Z\Lambda_L$ denotes the isomorphism in \Cref{cor: module structure of log unit lattice}, then the restriction $\phi_v\restriction_{R_\Q[G]}$ yields an isomorphism of $\Q$-vector spaces between $\Q\otimes\Lambda_K$ and the right-ideal $N_HR_\Q[G]$. It follows that if $B$ is a $\Q$-basis of $N_HR_\Q[G]$, then the set $\{\alpha v\}_{\alpha\in B}$ is a $\Q$-basis for $\Q\otimes \Lambda_K$.
    \item[(ii)] If $K/\Q$ is Galois, then $N_H$ is a central element in $R[G]$ and $N_HR[G]\cong R[G/H]$ as $R[G]$-modules. Furthermore, the unit $u^{N_H}=N_{L/K}(u)$ is a weak Minkowski unit of $K$.
\end{enumerate}
    \begin{proof}
	Note that the isotypic component of the group algebra $\Q[H]$ corresponding to the trivial representation of $H$, is the ideal generated by the central idempotent $N_H/\#H$. We conclude that when we view $R_\Q[G]$ and $\Q\otimes_\Z\Lambda_L$ as left $\Q[H]$-modules we must have:
	\[
	R_\Q[G]^{H}=N_H R_\Q[G]\quad \text{and}\quad (\Q\otimes\Lambda_L)^H=N_H(\Q\otimes_Z \Lambda_L).
	\]
    Clearly $N_H\Lambda_L\subset \Lambda_K$ and $\Lambda_K\subset \Lambda_L^H$, hence $\Q\otimes_\Z\Lambda_K=(\Q\otimes_\Z\Lambda_L)^H$. Since $\phi_v$ maps $N_HR_{\Q}[G]$ isomorphically to $N_{H}(\Q\otimes_\Z \Lambda_L)$, the result follows.
\end{proof}
\end{proposition}

\subsection{$G$-invariant symmetric bilinear forms on the ring $R[G]$ }\label{subsection: symmetricbil}
Let $F$ be a field of characteristic $0$. We now establish some facts about $G$-invariant symmetric bilinear forms on the semisimple algebra $R_F:=F[G]/(N_G)$.
\begin{definition}
Let $V_F$ a finite dimensional $F$-vector space and $\rho:G\rightarrow \mathrm{GL}(V)$ be a representation.
We say that a symmetric $F$-bilinear form $\mathfrak{b}:V\times V\rightarrow F$ is $G$-invariant if 
\[
\mathfrak{b}(\rho(g)\alpha,\rho(g)\beta)=b(\alpha,\beta),\quad\forall \alpha,\beta\in V_F.
\]
We define $\mathrm{Sym}^G(V_F)$ to be the set of symmetric $G$-invariant $F$-bilinear forms on $V$.
\end{definition}
Note that $\mathrm{Sym}^G(V_F)$ is an $F$-vector space, and it is also an irreducible affine algebraic variety over $F$; it is the zero locus of the linear equations defined by the `$G$-invariance' condition on the space of symmetric matrices with coefficients in $F$. The following lemma relates the isotypic decomposition of a representation with the structure of the space of $G$-invariant bilinear forms on it.
\begin{lemma}\label{lemma: isotypic comp are orthogonal}
     Let $V_F$ be a finite dimensional representation of $G$.
     Let $e\in \Q[G]$ be a central idempotent.
     For any symmetric $G$-invariant bilinear form $\mathfrak b\in \mathrm{Sym}^G(V_F)$, the vector spaces $eV_F$ and $(1-e)V_F$ are orthogonal with respect to $\mathfrak b$. Moreover, the choice of a basis for $eV_F$ and $(1-e)V_F$ produces an isomorphism of algebraic varieties: \[ 
    \mathrm{Sym}^G(V_{F})\cong \mathrm{Sym}^G(eV_F)\times\mathrm{Sym}^G((1-e)V_F).
    \]
\end{lemma}
\begin{proof}
    Since $e\in \Q[G]$, we have $\overline{e}=e$, where
     \[
     \overline{(\cdot)}:\Q[G]\rightarrow \Q[G]
     \]
     is the $\Q$-linear map determined by $\overline{g}=g^{-1}$.
    Note that for any $b\in\mathrm{Sym}^G(R_F)$, the group $G$ acts orthogonally (with respect to $b$) on $R_F$. Hence, for any $\alpha,\beta\in V_F$, we get: \[
    \mathfrak b( e\alpha,(1-e)\beta)=\mathfrak{b}( \alpha, \overline{e}(1-e)\beta)=\mathfrak b(\alpha, e(1-e)\beta\rangle =0.
    \]
    This proves that $eV_F$ and $(1-e)V_F$ are orthogonal with respect to $\mathfrak b$. The rest of the proof follows easily from the decomposition of $R_F$-modules:
    \[
    V_F\cong eV_F\oplus (1-e)V_F.
    \]
 
\end{proof}    
From now on, we restrict ourselves to the case $V_F=R_F$. The algebra $R_F$ has an involution $\overline{(\cdot)}:R_F\rightarrow R_F^{opp}$ induced by the map $G\rightarrow G$, $g\mapsto g^{-1}$; we denote by $A_F$ the $F$-vector subspace invariant under $\overline{(\cdot)}$, so
\[
A_F:=\{\eta\in R_F:\overline{\eta}=\eta\}.
\]
Note that we can view $A_F$ a closed affine irreducible subvariety of $R_F$.
\begin{lemma}[{\cite[Proposition 12]{bayer2002ideal}}]\label{lemma: traceforms}
    Let $F$ be a field of characteristic 0. For $\alpha\in A_F$, let $\mathrm{Tr}_\alpha:R_F\times R_F\rightarrow F$ be the bilinear form given by $\mathrm{Tr}_\alpha(x,y)=\mathrm{trace}_F(x\alpha\overline{y})$, where we identify $x\alpha\overline{y}$ with the linear map given by multiplication on the left by $x\alpha\overline{y}$. Then the map
    \[
    A_F\rightarrow \mathrm{Sym}_F(R_F),\quad \eta\mapsto\left\{(\alpha,\beta)\mapsto\mathrm{Tr}\left(\alpha\eta\overline{\beta
    }\right)\right\}
    \]
    is an isomorphism of $F$-vector spaces and also an isomorphism of $F$-varieties.
\end{lemma}
\begin{proof}
    Note that $R_F$ is a separable algebra, hence the trace form $\mathrm{Tr}_1$ is a non-degenerate bilinear form. Since the bilinear form $B:R_F\times R_F\rightarrow F$ is $G$-invariant and symmetric, it follows that:
    \[
    B(x,y)=B(1,\overline{x}y).
    \]
    Let $R_F^*$ be the dual $F$-vector space of $R_F$ and consider the linear functional:
    \[
    R_F\rightarrow F,\quad x\mapsto B(1,\overline{x}).
    \]
    Since the bilinear form $\mathrm{Tr}_1$ is non-degenerate, there exists a unique $\alpha\in R_F$, such that
    \[
    \mathrm{Tr}_1(\alpha,(\cdot))=B(1,\overline{(\cdot)}).
    \]
    It follows that 
    \[
    \mathrm{Tr}_1(\alpha,\overline{x}y)=B(1,\overline{x}y).
    \]
    Hence
    \[
    \mathrm{Tr}_{\alpha}(x,y)=B(x,y).
    \]
    This proves the isomorphism of the linear map, and this gives the result.
\end{proof} 
\begin{proposition}[Compare with Theorem 42 from \cite{gargava2023lattice}]\label{prop: Cholesky decomposition}
    Let $F$ be an algebraicallly closed field of characteristic 0. The map
    \[
    \psi: R_F\rightarrow A_F,\quad \nu\mapsto \nu\overline{\nu}
    \]
    is surjective on $F$-points.
\end{proposition}
\begin{proof}
    Let $\eta\in A_F$. Consider the decomposition of $R_F$ into central simple algebras: \[
    R_F=B_1\oplus\dots\oplus B_k,
    \]
    and write $\eta=x_1+\dots+ x_{k}$, where $x_i\in B_i$, for $i=1,\dots,k$.
    Note that the involution $\overline{(\cdot)}$ induces a permutation $\sigma\in S_{k}$ of the above direct summands. It follows that $\mathrm{proj}_{i}(\overline{\eta})\in B_{\sigma(i)}$, and since $\overline{\eta}=\eta$, we obtain:
    \[
    \mathrm{Proj}_i(\eta)=x_{\sigma(i)}.
    \]
    Let us now construct an element $\nu\in R_F$ such that $\eta=\overline{\nu}\nu$. If $\sigma(i)\neq i$, set $z_i=x_i$ and set $z_{\sigma(i)}$ as the identity element of the algebra $B_{\sigma(i)}$. If $\sigma(i)=i$, then, for this fixed index $i$, we have that $\overline{(\cdot)}$ induces an involution of the algebra $B_i$. The proof will follow if we find $z_i\in B_i$ such that $z_i\overline{z_i}=x_i$. To simplify the notation, from now on we will omit the subindex $i$ and write $B=B_i, x=x_i$ and $n=n_i$. By the Artin-Wedderburn theorem, $B_i$ is isomorphic to the algebra of matrices $M_{n_i}(F)$, for some $n_i\in \Z$. We may now apply \cite[Theorem 7.4]{scharlau2012quadratic}, to get that there exists $w\in M_{n}(F)$ such that the matrix $w$ is symmetric or antisymmetric, and
    \[
    \overline{x}=wx^{tr}w^{-1}.
    \]
    We shall only complete the proof for the antisymmetric case, since the symmetric case can be handled in an analogous manner. Note that we are assuming $\overline{x}=x$, hence the above becomes:
    \[
    x=wx^{tr}w^{-1}.
    \]
      Since $w$ is antisymmetric and invertible, we have that $n$ is even, so there exists an integer $m$ such that $n=2m$. Moreover, if $J$ is the matrix that has $m$ blocks of the form $\begin{bmatrix}
        0&1\\
        -1&0
    \end{bmatrix}$ along the main diagonal and 0's elsewhere:
   \[
        J=\begin{bmatrix}
            0&1\\
            -1&0\\
                &&0&1\\
                &&-1&0\\
                &&&&\ddots\\
                &&&&&&0&1\\
                &&&&&&-1&0
        \end{bmatrix}
       \] 
     then there exists $p$ such that
    \[
    w=p^{tr}Jp.
    \]
    Now choose a matrix $u$ such that
    \[
    x=p^{tr}uJ(p^{tr})^{-1}.
    \]
    From the fact that $x=\overline{x}$, one sees that $u$ is antisymmetric, so there exists a matrix $s$ such that $u=sJ's^{tr}$ where $J'$ is the matrix analogous to $J$, which has $l$ blocks, ($l\leq m),$ of the form $\begin{bmatrix}
        0&1\\
        -1&0
    \end{bmatrix}$ along the main diagonal and 0's elsewhere. Note that we may even choose $s$, so that $sJs^{tr}=sJ's^{tr}$.  Let $i$ be a choice of a square root of $-1$.  If we let $z=i\,p^{tr}s(p^{tr})^{-1}$, then it is not hard to compute:
    \[
    z\overline{z}=x.
    \]
    This concludes the result.
\end{proof}

\section{The genericity of the log unit lattice}\label{section: genericity}
For the rest of this section, $\overline{\Q}\subset \C$ denotes the algebraic closure of $\Q$ inside $\C$.
The main objective of this section is to prove \Cref{theorem: log-lats are generic}; it indicates that the Gram matrix of a log unit lattice of a real Galois number field with Galois group $G$, corresponds to the generic point of the irreducible algebraic variety $\mathrm{Sym}^G(R_{\overline{\Q}}[G])$, which is a fundamental result for the proofs of both \Cref{Theorem: CompleteInvariants} and \Cref{Theorem: RegulatorArithmeticInvariant}. We start by giving a definition that will help to make the above statement more precise.
\begin{definition}\label{definition: generic}
    Let $X_{\overline{\Q}}$ be an irreducible algebraic variety defined over $\overline{\Q}$, let $X_\C$ be its extension of scalars to $\C$, and let $\beta:X_\C\rightarrow X_{\overline{\Q}}$ be the base change map. An element $p\in X_\C$ is $\overline{\Q}$-generic if $\{\beta(p)\}$ is dense in $X_{\overline{\Q}}$ with respect to the Zariski topology.
\end{definition}

Theorem \ref{theorem: log-lats are generic} is conditional on the \emph{algebraic independence of logarithms}:
\begin{conjecture}[Conjecture 14.1 of \cite{waldschmidt1992linear}]\label{conjecture: independence of logs}
    Let $\ell_1,\dots,\ell_m\in\C$ be $\Q$-linearly independent logarithms of algebraic numbers, then $\ell_1,\dots,\ell_m$ are algebraically independent.
\end{conjecture}
Let us now set the framework for the sequel; from now on, $K$ denotes a real Galois number field of degree $n:=[K:\Q]$, and $G:=\Gal(K/\Q)$ denotes its Galois group. The letter $F$ is reserved to denote a field of characteristic 0, and we will use the notation $R_F$, and $\mathrm{Sym}^G(R_F)$ from \Cref{subsection: symmetricbil}. We let $\mathcal{H}_F\subset \mathbb{A}_F^n$ denote the algebraic variety given by the hyperplane in the $n$-dimensional affine $F$-space satisfying the equation $\sum x_i=0$. The action of $G$ on $\Lambda_K$ endows $\mathcal{H}_F$ with an orthogonal action with respect to the usual dot product. We now use $\mathcal{H}_F$ to parametrize bilinear forms in $\mathrm{Sym}^G(R_F)$ by vectors in $\mathcal{H}_F$.

\begin{lemma}\label{lemma:grammatrixMinkunit}
     Let $F$ be an algebraically closed field. The map
    \[
    \mathfrak{Gr}:\mathcal{H}_F\rightarrow \mathrm{Sym}^G(R_F),\quad v\mapsto \{(\alpha,\beta)\mapsto \langle \alpha v,\beta v\rangle \}
    \]
    is surjective on $\overline{F}$-points.
\end{lemma}
\begin{proof}
    Let $\mathfrak{b}\in \mathrm{Sym}^G(R_{F})$ by Lemma \ref{lemma: traceforms}, there exists $\eta\in A_F$ such that 
    \begin{equation}\label{eq: bilinearvstr}
    \mathfrak{b}(\alpha,\beta)=\mathrm{Tr}(\alpha\eta\overline{\beta}).
    \end{equation}
    Choose an isomorphism $\varphi:\mathcal{H}_F\rightarrow R_F$ of $R_F$-modules, let $w=\varphi^{-1}(1)$ and apply Lemma \ref{lemma: traceforms}  again to obtain $\gamma\in A_F$ such that
    \begin{equation}\label{eq: innerproducvstrace}
    \langle \alpha w,\beta w\rangle=\mathrm{Tr}(\alpha \gamma \overline{\beta}),\quad\forall \alpha,\beta\in R_F.
    \end{equation}
    Note that $\gamma$ is invertible since $\langle\cdot,\cdot\rangle$ is non-degenerate. Now, by Cholesky's decomposition, there exist $\nu'$ and $\nu''\in R_F$ such that:
    \[
    \gamma=\nu'\overline{\nu'}\quad\text{and}\quad \eta=\nu''\overline{\nu''}.
    \]
    It follows that
    \begin{equation}\label{eq: eta in terms of gamma}
    \eta=\nu''\nu'^{-1}\,\gamma\, \overline{\nu'^{-1}}\,\overline{\nu''}.
    \end{equation}
    Hence, by \ref{eq: innerproducvstrace} and \ref{eq: bilinearvstr}, we get that for every $\alpha,\beta\in R_{F}$:
    \begin{align*}
        \mathfrak{b}(\alpha,\beta)&=\mathrm{Tr}(\alpha\eta\overline{\beta})\\
        &=\mathrm{Tr}(\alpha\nu''\nu'^{-1}\,\gamma\, \overline{\nu'^{-1}}\,\overline{\nu''}\overline{\beta}),\quad \text{(from \ref{eq: eta in terms of gamma})}.\\
        &=\langle\alpha (\nu''\nu'^{-1}w),\beta(\nu''\nu'^{-1}w)\rangle.
    \end{align*}
    It follows that $\mathfrak{Gr}( \nu''\nu'^{-1}w)=\mathfrak{b}$. This proves the result.
\end{proof}
We now show that there exist units in $\E_K^\times$ whose corresponding log-vectors are not contained in a Zariski closed subset definable over $\overline{\Q}$ in $\mathcal{H}_{\C}$.
\begin{lemma}\label{Lemma: generic unit}
    Let $\mathcal{H}_{\overline{\Q}}\subset \mathbb{A}_{\overline{\Q}}^n$ be the the algebraic variety given by the hyperplane in the\, $n$-dimensional $\overline{\Q}$-affine space satisfying the equation $\sum x_i=0$. Assume \Cref{conjecture: independence of logs}. Let $u\in \E_{K}^\times$ be a weak Minkowski unit, then $\mathrm{Log}(u)\in \mathcal{H}_{\C}$ is $\overline{\Q}$-generic.
\end{lemma}
\begin{proof}
Let $\beta_{\mathcal H}:\mathcal{H}_\C\rightarrow \mathcal{H}_{\overline{\Q}}$ be the base change map. Choose a subset $A\subset G$ with $n-1$ elements. Since $K$ is totally real and since $u$ is a weak Minkowski unit, for any embedding $\iota:K\rightarrow \R$, the subset \[\{\log|\iota(g(u))|\}_{g\in A}\subset \R\] is $\Q$-linearly independent. It follows that the first $n-1$ entries of the vector $\mathrm{Log}(u)$ are algebraically independent. In other words, the point $\beta_{\mathcal H}(\mathrm{Log}(u))$ is not contained in any Zariski closed proper subset of $\mathcal{H}_{\overline{\Q}}$. The result follows.
\end{proof}

\begin{theorem}\label{theorem: log-lats are generic}
    Assume Conjecture \ref{conjecture: independence of logs}. Let $K/\Q$ be a real Galois number field with Galois group $G$ and let $\Lambda_K$ be its log unit lattice. Given $F/\Q$ an extension of fields,
    let $\mathrm{Sym}^G(R_F[G])$ be the $F$-variety consisting of $G$-invariant $F$-valued bilinear forms on $R_F[G]$. Suppose $u\in \E_K^\times$ is a weak Minkowski unit (see \Cref{Theorem: existence of weak Minkowski unit}). If $v=\Log(u)$, then the bilinear form
    \[
    \mathfrak{Gr}_v:R_\C\times R_\C\rightarrow \C, \quad \mathfrak{Gr}_v(\alpha,\beta)=\langle \alpha v,\beta v\rangle,
    \]
    is $\overline{\Q}$-generic.
\end{theorem}
\begin{proof}
Recall the map $\mathfrak{Gr}_F:\mathcal{H}_\C\rightarrow \mathrm{Sym}^G(R_F)$ from \Cref{lemma:grammatrixMinkunit}. Consider the commutative diagram given by base change:
\[\begin{tikzcd}
	{\mathcal{H}_\C} & {\mathrm{Sym}^G(R_\C)} \\
	{H_{\overline{\mathbb{Q}}}} & {\mathrm{Sym}^G(R_{\overline{\Q}})}
	\arrow["{\mathfrak{Gr}_\C}", from=1-1, to=1-2]
	\arrow["\beta"', from=1-1, to=2-1]
	\arrow[from=1-2, to=2-2]
	\arrow["{\mathfrak{Gr}_{\overline{\Q}}}", from=2-1, to=2-2]
\end{tikzcd}\]
By Lemma \ref{lemma:grammatrixMinkunit}, the map $\mathfrak{Gr}_{\overline{\Q}}$ is a surjective map on $\overline{\Q}$ points, hence it is a surjective map of $\overline{\Q}$-schemes. Now, use Lemma \ref{Lemma: generic unit} and the fact that $u\in \E_K^\times$ is a weak Minkowski unit to get that $\mathrm{Log}(u)$ corresponds to a closed point in $\mathcal{H}_\C$ that is $\overline{\Q}$-generic. Since the image of a dense subset under a surjective map is dense, the set $\{\mathfrak{Gr}_{\overline{\Q}}(\beta(\mathrm{Log}(u)))\}$ is dense in the variety $\mathrm{Sym}^G(R_\Q)$. This proves the result.
\end{proof}
We finish this section with a remark that relates the log unit lattice with the bilinear form from \Cref{theorem: log-lats are generic}
\begin{remark}\label{remark: Gram vs weakGram}
	Let $K$ be any real Galois number field of degree $[K:\Q]=n$ with Galois group $G:=\Gal(K/\Q)$, let $\Lambda_K$ be its log unit lattice and let $b_K$ be any Gram matrix for $\Lambda_K$. Let $u$ be a weak Minkowski unit, let $v=\Log(u)$, and let $\mathrm{Gr}_v$ be any Gram matrix (with respect to a $\Q$-basis of $R_\Q$) of the bilinear form:
	\[
	\mathfrak{Gr}_v:R_\C\times R_{\C}\rightarrow \C\quad \mathfrak{Gr}_v(\alpha,\beta)=\langle \alpha v,\beta v \rangle.
	\]
     There exists a matrix $A\in \mathrm{GL}_{n-1}(\Q)$ such that
     \[A\mathrm{Gr}_vA^{tr}=b_K.
     \]
     This follows from the fact that $\Q\otimes_\Z\Lambda_K\cong R_\Q.$
\end{remark}

\section{The linear independence at $s=1$ of distinct Dedekind zeta functions}\label{section: linear independence}
This section is dedicated to the proof of \Cref{Theorem: RegulatorArithmeticInvariant}, which states, assuming \Cref{conjecture: independence of logs}, the linear independence of the residues at $s=1$ of distinct Dedekind zeta functions. We start by citing a special case of a result from \cite{DeyGorlachKaihnsa2020} that we will use in the sequel. Afterwards, we will introduce the classical results of Gassman on arithmetic equivalence, these results allow us to reduce the proof to a representation theory statement. With these Tools in hand, we introduce some notation and then proceed with the proof of \Cref{Theorem: RegulatorArithmeticInvariant}. 
\subsection{Coordinate-wise square map in a hyperplane}
Let $F$ be an algebraically closed field of characteristic 0. Given $e\geq 1$, we denote by $F[x_0,\dots, x_n]_e$ the vector space of polynomials of degree $e$. Consider the action of the group $(\Z/2)^{n+1}$ on $F[x_1,\dots,x_n]_e$ given by rescaling the coordinates $x_0,\dots,x_n$ with $\pm1$. Let $\mathcal{G}_{n+1}$ the quotient $(\Z/2)^{n+1}$ by the diagonal subgroup $\{(g,\dots,g):g\in \Z/2\}$. The above defines an action of $\mathcal{G}_2$ on the projective space $\mathbb{P}(F[x_0,\dots, x_n]_e)$. The following Lemma is \cite[Lemma 3.3 and Proposition 3.4]{DeyGorlachKaihnsa2020}  for the special case of the "squaring map" on hyperplanes.
\begin{lemma}\label{lemma: hypersurface}
	Let $\mathfrak{h}\subset \mathbb P^n_{\overline{\Q}}$ be a hyperplane in the projective space given by the zero set of a linear form: $f(x_0,\dots,x_n)=\sum_{i=0}^{n} c_i x_i=0$, $c_i\in \overline{\Q}$. Let $\varphi_2$ be the map:
	\[
	\varphi_2:\mathfrak{h}\rightarrow \mathbb{P}^n,\quad [x_0:\cdots:x_n]\mapsto [x_0^2:\hdots:x_n^2]
	\]
    \begin{enumerate}
    \item[(i)] Consider the polynomial
    \[
    h(x_0,\dots,x_n)=\prod_{g\in \mathcal{G}_2}g\cdot f(x_0,\dots,x_n).
    \]
    Then $h(x_0,\dots, x_n)\in \overline{\Q}[x_0^2,\dots, x_n^2]$.
    \item[(ii)] the image of $\varphi_2$ is a hypersurface given by the zero set of the polynomial $f(x_0,\dots,x_n)$ given by replacing every instance of $x_i^2$ in the polynomial $h(x_0,\dots, x_n)$.
    \end{enumerate}
  
 \end{lemma}

\subsection{Arithmetic equivalence and representation theory}
We will now introduce some fundamental results by Gassman that we need. Our main source for arithmetic equivalence is \cite{sutherland2018arithmetic}; see also \cite{MantillaSoler2019} for a perspective on arithmetic equivalence via the Artin formalism.

Let $K$ be a number field of degree $n=[K:\Q]$, let $N/\Q$ be any Galois number field containing $K$, let $G:=\Gal(N/\Q)$ and $H:=\Gal(N/K)$. From now on $R=\Z[G]/(N_G)$. Consider the set of all embeddings of $K$ into $\C$: 
\[E_K:=\hom_{\Q}(K,\C).\]
The action of $G$ on the set $E_K$ yields a permutation representation $T_K:=\Q^{E_K}$. On the other hand, if we let $N_H:=\sum_{h\in H}h\in \Q[G]$, the natural right action of $G$ on the right ideal $N_H\Q[G]$ also endows it with the structure of a subrepresentation of $\Q[G]$.
The following proposition is essential for the proof of \Cref{Theorem: RegulatorArithmeticInvariant}. 
\begin{proposition}\label{prop: Gassman criterion}
    Let $K_1$ and $K_2$ be two number fields, let $N/\Q$ be a Galois number field containing $K_1$ and $K_2$. Let $G:=\Gal(N/\Q)$, $H_1:=\Gal(N/K_1)$ and $H_2:=\Gal(N/K_2)$. The following are equivalent:
    \begin{enumerate}
        \item[(a)] $K_1$ and $K_2$ are arithmetically equivalent;
        \item[(b)] $N_{H_1}\Q[G]\cong N_{H_2}\Q[G]$ as right $\Q[G]$-modules, and
        \item[(c)] $N_{H_1}R_\Q \cong N_{H_2}R_\Q$ as right $R_\Q$-modules.
    \end{enumerate}
\end{proposition}
\begin{proof}
    \textit The equivalence of \textit{(a)} with \textit{(b)} is well known. Let us prove that \textit{(b)} is equivalent to \textit{(c)}. Note that $e=N_G/(\#G)$ is the primitive central idempotent corresponding to the trivial representation in $\Q[G]$. Then, for $i=1,2$, we have the following isomorphisms of representations: 
    \[ N_{H_i}\Q[G]\cong\Q_{\mathrm{triv}}\oplus (1-e)N_{H_i}\Q[G],\quad\text{and}\quad N_{H_i}R_\Q\cong (1-e)N_{H_i}\Q[G],\]
    where $\Q_{\mathrm{triv}}$ is the trivial representation of $G$. Given that two representations are isomorphic if and only if their isotypic components are isomorphic, we have
    \[
    N_{H_1}\Q[G]\cong N_{H_2}\Q[G]\Leftrightarrow (1-e)N_{H_1}\Q[G]\cong (1-e)N_{H_2}\Q[G]. 
    \]
    It follows that for $i=1,2$, there is an isomorphism $(1-e)N_{H_i}\Q[G]\cong N_{H_i}R_\Q$. This gives the result.
\end{proof}

\subsection{Proof of the linear independence}
Let us start by introducing some notation. Given totally real number fields $K_1,\dots,K_n$, let $N$ be the Galois closure of the compositum $K_1,\dots, K_n$, and let $G:=\Gal(N/\Q)$. Let $[n]=\{1,\dots, n\}$. For $i\in [n]$, we let $H_i:=\Gal(N/K_i)$, set $m_i:=[G:H_i]$, and let $V_{i}:=N_{H_i}R_{\overline{\Q}}$. We are also going to consider the isotypic decompositions:
\begin{equation}\label{eq:isotypic decompositions}
R_{\overline{\Q}}=\bigoplus_{\pi\in\mathrm{Irrep}(G),\pi\neq 1}(R_\pi)_{\overline{\Q}},\quad \text{and}\quad V_i=\bigoplus_{\pi\in\mathrm{Irrep}(G),\pi\neq 1}(V_i)_\pi,\quad i\in [n].
\end{equation}
Let us also choose bases $B_\pi$ and $(B_i)_\pi$ for each of the above components so that:
\begin{equation}\label{eq: choice of basis isotypic comp}
(R_\pi)_{\overline{\Q}}=\mathrm{span}_{\overline{\Q}}(B)_\pi \quad\text{and}\quad V_i=\mathrm{span}_{\overline{\Q}}(B_i)_\pi,\quad \pi\neq1\in \mathrm{Irrep}(G).
\end{equation}

It follows that $B_i:=\bigcup_\pi (B_{i})_\pi$ is a basis of $V_i$ and $B:=\bigcup_\pi B_\pi$ is a basis of $R_{\overline{\Q}}$.

\begin{theorem}\label{Theorem: RegulatorArithmeticInvariant}
   Assume \Cref{conjecture: independence of logs} and let $K_1,\, K_2,\,\dots,\,K_n$ be totally real number fields all distinct and different from $\Q$ such that $\zeta_{K_i}\neq\zeta_{K_j}$ for every $i\neq j$. Then, the sets  
     \[\mathfrak{reg}:=\{\reg(K_i)\}_{i=1}^n\quad \text{and}\quad \mathfrak{res}:=\{\mathrm{res}_{s=1}\zeta_{K_i}(s)\}_{i=1}^n\] are linearly independent subsets of $\C$ over $\overline{\Q}$.  
\end{theorem}
\begin{proof}
    From now on we will use the notation $[n]:=\{1,\dots,n\}$. Let us prove by induction on $n$ that the sets $\mathfrak{res}$ and $\mathfrak{reg}$ are $\overline{\Q}$-linearly independent. The base case $n=1$ is clear. So, assume now that the result is true for any $k<n$, and let us prove the result for $n$. By the class number formula, it is clear that the $\overline{\Q}$-linear independence of $\mathfrak{reg}$ implies the $\overline{\Q}$-linear independence of $\mathfrak{res}$. So, it is enough to prove the $\overline{\Q}$-linear independence of the set $\mathfrak{reg}$. Suppose there is a linear combination
	\begin{equation}\label{eq: lineardependence}
	\sum_{i=1}^n\lambda_i\reg(K_i)=0. 
	\end{equation}

    Recall the isotypic decompositions in \cref{eq:isotypic decompositions}.
    Our objective is to prove that if the linear combination in \cref{eq: lineardependence} is nontrivial, then at least two of the representations $V_{1},\dots, V_n$ are isomorphic as $\overline{\Q}[G]$-modules which will imply that they are isomorphic as $\Q[G]$-modules, and then lead to the conclusion that at least two of the fields $K_1,\dots, K_n$ have the same Dedekind zeta function by \Cref{prop: Gassman criterion}. 

   
   For $i\in [n]$, and for each $\pi\in\mathrm{Irrep}(G)$, let $d_\pi^{(i)}:=\dim_{\overline{\Q}}(V_i)_\pi$. Note that for a fixed pair $i,j\in [n]$, the quantities $d_\pi^{(i)}$ and $d_\pi^{(j)}$ cannot be all the same for every $\pi$, otherwise we would have $V_i\cong V_j$ which would lead to $\zeta_{K_i}=\zeta_{K_j}$. 
    
    For $\pi\in \mathrm{Irrep}(G)$ and for $i\in [n]$, consider the bases $(B_{i})_\pi$ and $B_\pi$ of $(V_i)_{\pi}$ and $(R_{\overline{\Q}})_\pi$, respectively, from \cref{eq: choice of basis isotypic comp}. For a field $F$ of characteristic 0, define, for each $i\in [n]$, the following map of $F$-varieties: \[\psi_i:\mathrm{Sym}^G(R_{F})\rightarrow \mathbb{A}_{F}^1,\quad \mathfrak{b}\mapsto \det\left(\mathfrak{b}( w_k^i,\,w_l^i)\right)_{w_k^i,w_l^i\in B_i}.\]
    All these maps together yield the rational map: 
            \[
            \Psi_{F}:\mathrm{Sym}^G(R_{F})\rightarrow \mathbb{P}_{F}^n,\quad \Psi_{F}(\mathfrak{b})=[\psi_1(\mathfrak b):\psi_2(\mathfrak b):\dots:\psi_n(\mathfrak b)].
            \]
    The map $\Psi_{F}$ is only undefined for the closed subset consisting of bilinear forms $\mathfrak{b}\in \mathrm{Sym}^G(R_F)$ for which $\Psi_i(\mathfrak b)=0$ for all $i\in [n]$.
    
    Choose a weak Minkowski unit $u\in \E_N^\times$, let $v=\Log(u)$, and let $\mathfrak{Gr}_v\in\mathrm{Sym}^G(R_\C)$ be the bilinear form:\[
    \mathfrak{Gr}_v:R_\C \times R_\C\rightarrow \C, \quad  (\alpha,\beta)\mapsto \langle \alpha  v,\beta v\rangle.
    \]
    
     Note that by \Cref{prop: norm of weak Mink unit} and \Cref{remark: Gram vs weakGram} there exist nonzero constants $a_1,a_2,\dots, a_n\in\overline{\Q}$ such that
    \[
    \Psi_\C(\mathfrak{Gr}_v)=[a_1\reg(K_1)^2:a_2\reg(K_2)^2:\dots:a_n\reg^2(K_n)].
    \]
    Let $c_i=\frac{\lambda_i}{\sqrt{a_i}}$ and recall the group $\mathcal{G}_n$ from \Cref{lemma: hypersurface}. Then, \Cref{lemma: hypersurface} and \cref{eq: lineardependence} imply that  $\Psi(\mathfrak{Gr}_v)$ lies in the hypersurface $S$ given by the zero set of the polynomial $f(x_1,\dots,x_n)$ that arises by replacing each instance of $x_i^2$ with $x_i$ in the polynomial:
    \[
    \prod_{g\in\mathcal{G}_2}(c_1g_1 x_1+c_2g_2x_2+\dots+c_ng_n x_n)\in \overline{\Q}[x_1^2,\dots,x_n^2].
    \]
    Now consider the following commutative diagram given by base change from $\overline{\Q}$ to $\C$:
    \[\begin{tikzcd}
    	{\mathrm{Sym}^G(R_\C)} & {\mathbb{P}_\C^n} \\
    	{\mathrm{Sym}^G(R_{\overline{\Q}})} & {\mathbb{P}^n_{\overline{\Q}}}
    	\arrow["{\Psi_\C}", from=1-1, to=1-2]
    	\arrow["\beta", from=1-1, to=2-1]
    	\arrow[from=1-2, to=2-2]
    	\arrow["{\Psi_{\overline{\Q}}}", from=2-1, to=2-2]
    \end{tikzcd}\]
    The commutativity of the diagram indicates that $\Psi_\C(\beta(\mathfrak{Gr}_v))$ is contained in the hypersurface $S$. On the other hand, by \Cref{theorem: log-lats are generic},  the set $\{\Psi_{\overline{\Q}}(\beta(\mathfrak{Gr}_v))\}$ is dense in the image of $\Psi_{\overline{\Q}}$ and since $\Psi_{\overline{\Q}}(\beta(\mathfrak{Gr}_v))$ is contained in the closed subvariety $S$, it follows that $\Psi_{\overline{\Q}}(\mathfrak b)\in S$ for every $\mathfrak b\in\mathrm{Sym}^G(R_{\overline{\Q}})$ in the domain of $\Psi_{\overline{\Q}}$.
    
    Let $\mathfrak b\in\mathrm{Sym}^G(R_{\overline{\Q}})$ be an arbitrary bilinear form and set $(y_1,\dots,y_n)=\Psi_{\overline{\Q}}(\mathfrak{b}).$ Given that distinct isotypic components are orthogonal with respect to any bilinear form, we may decompose $\mathfrak b$ in the following way: for each nontrivial $\pi\in\mathrm{Irrep}(G)$ let $\mathfrak b_\pi\in\mathrm{Sym}^G(R_{\overline{\Q}})$ be the bilinear form that equals $\mathfrak{b}$ on $R_\pi$ and that is identically $0$ on $R_{\pi'}$ for every $\pi'\neq \pi$. Then $\mathfrak{b}:=\sum_{\pi}\mathfrak{b}_\pi$. 
    
    Let us now consider the following modification of $\mathfrak{b}$.
    Given $\pi'\in \mathrm{Irrep}(G)$ and $c\in\overline{\Q}$, let
    \[
    \mathfrak b(\pi',c):=c\mathfrak b_{\pi'}+\sum_{\pi\neq\pi'}\mathfrak b_{\pi}
    \]
    Then, since the determinant is a homogeneous polynomial, we have the formula
    \[
    \Psi_{\overline{\Q}}(\mathfrak b(\pi',c))=[c^{d_{\pi'}^{(1)}}y_1,c^{d_{\pi'}^{(2)}}y_2,\dots, c^{d_{\pi'}^{(n)}}y_n].
    \]
    Let \[d_{\pi'}:=\max_{i\in\{1,\dots,n\}}d_{\pi'}^{(i)}\quad\text{and}\quad D_{\pi'}:=\{i\in \{1,\dots,n\}:d_{\pi'}^{(i)}=d_{\pi'}\}.\]
    
    Note that the set $[n]\setminus D_{\pi'}$ is empty if and only if all the isotypic components associated to $\pi'$ of the representations $V_i$ are isomorphic, hence there exists an irreducible representation $\pi'$ for which the set $[n]\setminus D_{\pi'}$ nonempty. From now on we will fix $\pi'$ with this property. Observe also that $\Psi_{\overline{\Q}}(\mathfrak b(\pi',c))$ is contained in $S$. It follows that the polynomial in the variable $c$:
    \[
    f(\Psi_{\overline{\Q}}(\mathfrak b(\pi',c)))=\prod_{g\in\mathcal G_2}\left(\left(\sum_{i\in D_{\pi'}}c_i g_i \sqrt{y}_i\right)c^{d_{\pi'}}+\sum_{i\in[n]\setminus D_{\pi'}}(c_i g_i \sqrt{y}_i)c^{d_{\pi'}^{(i)}}\right)
    \]
   is identically zero. But this only holds if all the coefficients of all the powers of $c$ vanish. In particular, we must have
    \begin{equation}\label{eq: newlinearrelation}
    \sum_{i\in D_{\pi'}}c_ig_i \sqrt{y_i}=0,\quad \text{for some $g\in \mathcal{G}_2$}.      
    \end{equation}
    Now let $m:=\#D_{\pi'}$ and let $\mathcal{G}_2'$ be the quotient group $(\Z/2)^m$ by its diagonal subgroup. The group $\mathcal{G}_2'$ acts on $\mathbb{P}(\C[\{x_i\}_{i\in D_{\pi'}}]_e)$ as in \Cref{lemma: hypersurface}.
    Since \cref{eq: newlinearrelation} holds for arbitrary $\mathfrak b\in \mathrm{Sym}^G(R_{\overline{\Q}})$, we obtain that for every $\mathfrak b\in \mathrm{Sym}^G(R_{\overline{\Q}})$, the point $\psi_{\overline{\Q}}(\mathfrak b)$ is a zero of the polynomial $f_{\pi'}(\{x_i\}_{i\in D_{\pi'}})$ defined by replacing every instance of $x_i^2$ by $x_i$ in the polynomial:
    \[
    \prod_{g\in \mathcal G_2'}\sum_{i\in D_{\pi'}}c_ig_i x_i \in \overline{\Q}[\{x_i^2\}_{i\in D_\pi'}].
    \]
     We conclude that $f_{\pi'}$ vanishes in the whole image of $\Psi_{\overline{\Q}}$ and $\Psi_{\C}$. It follows that $f_{\pi'}(\Psi_{\C}(\mathfrak{Gr}(v)))=0$\, as well, and this gives that the set $\{\reg(K_i)\}_{i\in D_{\pi'}}$ is linearly dependent over $\overline{\Q}$. Given that $D_{\pi'}$ is a proper subset of $[n]$, the induction hypothesis implies $\lambda_{i}=0$ for all $i\in D_{\pi'}$. From here, we may apply the induction hypothesis again to the set $[n]\setminus D_{\pi'}$ to derive $\lambda_i=0$ for all $i\in [n]$. The result follows.
\end{proof}

\begin{corollary}\label{cor: Number fields with the same regulator}
    Assume \Cref{conjecture: independence of logs} and let $K_1$ and $K_2$ be two totally real number fields. If $\reg(K_1)$ and $\reg(K_2)$ are linearly dependent over the field of algebraic numbers $\overline{\Q}$, then $\zeta_{K_1}=\zeta_{K_2}$. Furthermore, $\reg(K_1)=\reg(K_2)$ if and only if $\zeta_{K_1}=\zeta_{K_2}$ and the class numbers $h_{K_1}$ and $h_{K_2}$ of $K_1$ and $K_2$, respectively, coincide.
\end{corollary}
\begin{proof}
    If $\reg(K_1)$ and $\reg(K_2)$ are linearly dependent, then \Cref{Theorem: RegulatorArithmeticInvariant} gives $\zeta_{K_1}$ and $\zeta_{K_2}$. In particular, we must have
    \[
    h_{K_1}\reg(K_1)=h_{K_2}\reg(K_2).
    \]
    The result follows easily from here.
\end{proof}


\section{On the log unit lattice as an invariant for number fields}\label{section: isometry and similarity}
We now investigate the strength of the isometry class and the similarity class of the log unit lattice as a classifying invariant for totally real number fields.
We now prove \Cref{Theorem: CompleteInvariants} conditionally on \Cref{conjecture: independence of logs}. This theorem states that two real Galois number fields with similar log unit lattices are isomorphic.
\subsection{Remarks on the isometry class of the log unit lattice of a totally real number field}
We now observe that two totally real number fields with isometric log unit lattices must have the same regulator.
\begin{corollary}\label{cor: isometry}
     Assume \Cref{conjecture: independence of logs}. Let $K_1$ and $K_2$ be two totally real number fields and let $\Lambda_{K_1}$ and $\Lambda_{K_2}$ be the log unit lattices of $K_1$ and $K_2$, respectively. If $\Lambda_{K_1}$ and $\Lambda_{K_2}$ are isometric, then $\reg(K_1)=\reg(K_2)$. In particular, $\zeta_{K_1}=\zeta_{K_2}$ and $h_{K_1}=h_{K_2}$, where $h_{K_1}$ and $h_{K_2}$ are the class numbers of the corresponding number fields.
\end{corollary}
\begin{proof}
    Note that if $K_1$ and $K_2$ are two totally real number fields with isometric log unit lattices $\Lambda_{K_1}$ and $\Lambda_{K_2}$, then we must have that $[K_1:\Q]=[K_2:\Q]$. On the other hand, given that for any totally real number field $K$ one has the formula
\[
\reg(K)=\vol(\Lambda_K)/\sqrt{n},
\]
it follows that $\reg(K_1)=\reg(K_2)$. \Cref{cor: Number fields with the same regulator} concludes the rest of the statement of the corollary.
\end{proof}
The following example can be found in \cite[Remark 2.11]{mantilla2015weak}.
\begin{example}\label{example: nonisometric}
    Consider the septic number fields
    \begin{equation}
    \begin{aligned}
        x^7-2x^6-47x^5+25x^4+755x^3+496x^2-3782x-5217\\
        x^7-2x^6-47x^5-8x^4+480x^3+793x^2+233x+19
    \end{aligned}
\end{equation}
Let $K_1$ and $K_2$ be the number fields defined by the above polynomials. Then $K_1$, $K_2$ are totally real, arithmetically equivalent and they have the same regulator. A computation in pari shows that the corresponding log unit lattices are not isometric, since they have different minimum. Given that these lattices also have the same covolume and different minimum the lattices are not similar either. It is worth noting that Mantilla-Soler proved that the corresponding integral trace forms are not isometric. Hence, the regulator cannot determine neither the isometry nor the similarity class of the log unit lattice, and it also cannot determine the integral trace form. 
\end{example}


\subsection{The similarity class of the log unit lattice of a real Galois number field}
Let us now set our framework for the proof of \Cref{Theorem: CompleteInvariants}. For the rest of this section,
$K_1$ and $K_2$ denote real Galois number fields, $L:=K_1\cdot K_2$ denotes the compositum of $K_1$ and $K_2$, and  $M:=K_1\cap K_2$. Let $G:=\Gal(L/\Q)$, $G_1:=\Gal(K_1/\Q)$ and $G_2:=\Gal(K_2/\Q)$. We will also need the Galois groups $H_1:=\Gal(L/K_1)$, $H_2:=\Gal(L/K_2)$, $T_1:=\Gal(K_1/M), T_2:=\Gal(K_2/M)$ and $S:=\Gal(M/\Q)$. 
The following diagrams illustrate the fields and Galois groups at play:
\begin{center}\begin{equation}\label{diagram: fields diagram}\begin{tikzcd}
	& L &&&& L \\
	{K_1} && {K_2} && {K_1} && {K_2} \\
	& \Q &&&& M \\
	&&&&& \Q
	\arrow["{H_1}"', no head, from=1-2, to=2-1]
	\arrow["{H_2}", no head, from=1-2, to=2-3]
	\arrow["G", no head, from=1-2, to=3-2]
	\arrow["{H_1}"', no head, from=1-6, to=2-5]
	\arrow["{H_2}", no head, from=1-6, to=2-7]
	\arrow["T", no head, from=1-6, to=3-6]
	\arrow["{G_1}"', no head, from=2-1, to=3-2]
	\arrow["{G_2}", no head, from=2-3, to=3-2]
	\arrow["{T_1}"', no head, from=2-5, to=3-6]
	\arrow["{T_2}", no head, from=2-7, to=3-6]
	\arrow["S", no head, from=3-6, to=4-6]
\end{tikzcd}\end{equation}\end{center}

Let us remark on how Galois subextensions yield central idempotents of $R_\Q[G]$.
\begin{remark}\label{remark: norm of mink is mink}
Let $T\subset G$ be a normal subgroup. Then $e_T:=N_T/\#T\in R_\Q[G]$ is a central idempotent, and $e_TR_\Q[G]\cong R_\Q[G/T]$ as $R_\Q[G/T]$-modules. Moreover, if $K=\mathrm{Fix}(T)$ is the fixed field of $T$, then, by \Cref{prop: norm of weak Mink unit}, if $u\in \E_{L}^\times$ is a weak Minkowski unit of $L$, then the unit $u^{N_T}=N_{L/K}(u)$ is a weak Minkowski unit of $K$.
\end{remark}

We now establish the main setup for the proof of \Cref{Theorem: CompleteInvariants}. Let $F$ be a field of characteristic 0, and let $K_1$, $K_2$, $L$ and $M$, $G$, $G_1$, $G_2$, $H_1$, $H_2$, $S$, $T_1$, and $T_2$ be as in \cref{diagram: fields diagram}. Consider the left $R[G]$-modules: 
\begin{equation}
  W_F:=N_{T}R[G]\quad\text{and}\quad(W_i^\perp)_F:= (1-N_{T_i}/\#T_i)R_{F}[G_i],\ \text{for $i=1,2.$} 
\end{equation}
 Also consider the left $R_F[G]$-module $V_F=(N_{H_1}+N_{H_2})R_F[G]$. Then $N_{H_1}V_F\cong R_F[G_1]$ and $W_F\cong R_F[S]$ as $R_F[G]$-modules. Using this isomorphism, we find that $T_1$ and $T_2$ act on $N_{H_1}V_F$ in a well defined manner. From this, it is easy to deduce the following direct sum decompositions.
 \begin{equation}\label{equation: decomposition}
 \begin{aligned}
    &V_F=R_F[S]\oplus (W_1^\perp)_F\oplus (W_2^\perp)_F\\
&R_F[G_i]=R_{F}[S]\oplus (W_i^\perp)_F,\quad (i=1,2).  
 \end{aligned}
 \end{equation}
 

We are now ready for proving the main result of this section.

\begin{theorem}\label{Theorem: CompleteInvariants}
    Let $K_1/\Q$ and $K_2/\Q$ be two real Galois number fields. For $i=1,2$, let $n_i:=[K_i:\Q]$ and $G_i:=\Gal(K_i/\Q)$. If $n_1,n_2>3$, then $[\Lambda_{K_1}]=[\Lambda_{K_2}]$\, if and only if $K_1\cong K_2$.
\end{theorem}
\begin{proof}
Let us first assume that $[\Lambda_{K_1}]=[\Lambda_{K_2}]$, then $K_1$ and $K_2$ must have the same degree, so $n_1=n_2$. Set
$L:=K_1\cdot K_2$ and $M=K_1\cap K_1$, and let  $G$, $G_1$, $G_2$, $H_1$, $H_2$, $S$, $T_1$, and $T_2$ be as in \cref{diagram: fields diagram}.

 Recall the decomposition from \ref{equation: decomposition}:
\[
V_F=R_F[S]\oplus(W_1^\perp)_F\oplus (W_2^\perp)_F.
\]
Now choose a weak Minkowski unit $u\in\E_L^\times$ of $L$, let $v=\Log(u)$ and consider the bilinear form $\mathfrak{Gr}_v:R[G]_\C\times R[G]_\C\rightarrow \C$ given by $\mathfrak{Gr}_v(\alpha,\beta)=\langle \alpha u,\beta u\rangle$.
Choose a basis $B_1,B_2$ and $B_M$ of $R_F[S]$, $W_F^\perp$ and $W_F^\perp$, respectively. Clearly, the union $B_1\cup B_2\cup B_M$ is a basis of $V_F$. Now, let
\[\begin{aligned}
\mathfrak{Gr}_M:=\mathfrak{Gr}_v\restriction_{ R_\C[S]\times R_\C[S]}\quad\text{and}\quad \mathfrak{Gr}_i=\mathfrak{Gr}_v\restriction_{(W_i^\perp)_F\times (W_i^\perp)_F},\quad(i=1,2).
\end{aligned}\]

Let us write $\mathrm{Gr}_1$, $\mathrm{Gr}_2$ and $\mathrm{Gr}_M$ for the corresponding Gram matrices. Note that by \Cref{remark: norm of mink is mink} the units
\[u^{N_{H_i}}=N_{L/K_i}\in \E_{K_i}^\times,\quad (i=1,2)\] 
are weak Minkowski units of $K_1$ and $K_2$, respectively. Furthermore
by \Cref{remark: Gram vs weakGram} and \Cref{lemma: including a log unit lattice}, for any Gram matrices $b_{K_1}$ and $b_{K_2}$ of $\Lambda_{K_1}$ and $\Lambda_{K_2}$, respectively, there exist matrices $C_1$ and $C_2$ and rational numbers $r_1, r_2\in \Q$ such that 
\[r_iC_i\begin{bmatrix}
	\mathrm{Gr}_i&\\&\mathrm{Gr}_M
\end{bmatrix}C_i^{tr}=b_{K_i},\quad (i=1,2).\]
 
It follows that the similarity of $\Lambda_{K_1}$ and $\Lambda_{K_2}$ gives the existence of a matrix $A\in GL_{n_1-1}(\Q)$ and a scalar $\lambda\in \R$ such that 
\[
\lambda     A\begin{bmatrix}
 \mathrm{Gr}_1&\\
    &\mathrm{Gr}_M
\end{bmatrix}A^{tr}=\begin{bmatrix}
 \mathrm{Gr}_2&\\
    &\mathrm{Gr}_M
\end{bmatrix}
\]
By solving for the scalar $\lambda$ in the above equation one obtains a rational algebraic relation between $\mathrm{Gr}_1$, $\mathrm{Gr}_2$, and $\mathrm{Gr}_M$, but since $(\mathrm{Gr}_1,\mathrm{Gr}_2,\mathrm{Gr}_M)\in\mathrm{Sym}^G(V_\C)$ is $\overline{\Q}$-generic, we must have $W_1=W_2=0$, or $W_M=0$. If $W_1=W_2=0$, then, by the definition of $V$ and $W$, we must have $K_1= K_2$ (viewed as subfields of $L$) and the result follows. So, suppose that $\mathrm{Gr}_M=0$, which implies $K_1\cap K_2=\Q$. It follows that
\[
\lambda A\mathrm{Gr}_1A^{tr}=\mathrm{Gr}_2.
\]
Given that the point $(\mathrm{Gr}_1,\mathrm{Gr}_2)\in \mathrm{Sym}^G(V_\C)$ is $\overline{\Q}$-generic and $A$ is an invertible rational matrix, we see that the variety $\mathrm{Sym}_\Q^{G_1}(W_1)\times\mathrm{Sym}_\Q^{G_2}(W_2)$ is two-dimensional at most. From this and from the fact that $M=\Q$, we obtain $[K_i:\Q]\leq 3$ for both $i=1,2$. The result follows.  
\end{proof}

\bibliographystyle{alpha}
\bibliography{sn-bibliography}

\end{document}